\documentclass[12pt]{amsart}
\pdfoutput=1

\usepackage[paperheight=11in,paperwidth=8.5in,left=1in,top=1.25in,right=1in,bottom=1.25in,]{geometry}
\usepackage[linktocpage=false,colorlinks=true,linkcolor=black,citecolor=black,
filecolor=black,urlcolor=black,pagebackref=false,pdfstartpage={1},
pdfstartview={FitH},pdftitle={},pdfauthor={},pdfsubject={},pdfcreator={},
pdfproducer={},pdfkeywords={}]{hyperref}
\usepackage{afterpage}
\usepackage{amsfonts}
\usepackage{amsmath}
\allowdisplaybreaks[4]
\usepackage{amssymb}
\usepackage{amsthm}
\usepackage{array}
\usepackage{bbm}
\usepackage[justification=centering]{caption}
\usepackage[capitalize,nameinlink,noabbrev,compress]{cleveref}

\usepackage{enumerate}
\usepackage{enumitem}
\usepackage{float}
\restylefloat{figure}
\usepackage{graphicx}
\usepackage{mathrsfs}
\usepackage[framemethod=tikz,ntheorem,xcolor]{mdframed}
\usepackage{parcolumns}
\usepackage{scalerel}
\usepackage{tabularx}
\usepackage{tikz}\pdfpageattr{/Group <</S /Transparency /I true /CS /DeviceRGB>>}
\usetikzlibrary{arrows,calc,cd,intersections,matrix,positioning,through}
\tikzset{commutative diagrams/.cd,every label/.append style = {font = \normalsize}}
\usepackage[all]{xy}
\usepackage[aligntableaux=center]{ytableau}

\numberwithin{equation}{section}
\newtheorem{thm}{Theorem}
\AfterEndEnvironment{thm}{\noindent\ignorespaces}
\numberwithin{thm}{section}

\AfterEndEnvironment{conj}{\noindent\ignorespaces}

\AfterEndEnvironment{cor}{\noindent\ignorespaces}
\newtheorem{lem}[thm]{Lemma}
\AfterEndEnvironment{lem}{\noindent\ignorespaces}

\AfterEndEnvironment{prob}{\noindent\ignorespaces}
\newtheorem{prop}[thm]{Proposition}
\AfterEndEnvironment{prop}{\noindent\ignorespaces}

\theoremstyle{definition}
\newtheorem{defn}[thm]{Definition}
\AfterEndEnvironment{defn}{\noindent\ignorespaces}
\newtheorem{eg_no_qed}[thm]{Example}
\newenvironment{eg}[1][]{\begin{eg_no_qed}[#1]\pushQED{\qed}}{\popQED\end{eg_no_qed}}
\AfterEndEnvironment{eg}{\noindent\ignorespaces}
\newtheorem{rmk}[thm]{Remark}
\AfterEndEnvironment{rmk}{\noindent\ignorespaces}

\theoremstyle{remark}

\AfterEndEnvironment{claim}{\noindent\ignorespaces}
\newtheorem*{claimpf_no_qed}{Proof of Claim}

\AfterEndEnvironment{claimpf}{\noindent\ignorespaces}

\renewcommand{\eqref}[1]{\hyperref[#1]{\textup{(\ref*{#1})}}}

\DeclareMathOperator{\Fl}{Fl}

\DeclareMathOperator{\GL}{GL}

\DeclareMathOperator{\Gr}{Gr}

\DeclareMathOperator{\spn}{span}

\newcommand{\transpose}[1]{{#1}{\hspace*{-0.2pt}\raisebox{4pt}{$\scriptstyle\mathsf{T}$}\hspace*{0.5pt}}}

\newcommand{\B}{\operatorname{B}}

\newcommand{\bip}[3][]{\mathsf{Q}^{#1}_{#2,#3}}
\newcommand{\Bminus}{\operatorname{B}^{-}}

\newcommand{\cell}[3][]{C^{#1}_{#2,#3}}
\newcommand{\dashedrightarrow}{\mathrel{\ThisStyle{\ooalign{$\SavedStyle\rightarrow$\cr\hfil\textcolor{white}{\rule{2\LMpt}{1\LMex}}\kern2\LMpt\hfil}}}}
\newcommand{\Diag}[1]{\hspace*{1pt}\mathsf{Diag}(#1)\hspace*{1pt}}

\renewcommand{\H}{\operatorname{T}}

\newcommand{\I}{I}

\newcommand{\N}{\operatorname{N}}

\renewcommand{\P}[2]{\operatorname{P}_{#1;\hspace*{0.5pt}#2}}
\newcommand{\PFl}[2]{\Fl_{#1;\hspace*{0.5pt}#2}}

\newcommand{\sumof}[1]{\sum\hspace*{-1pt}{#1}}

\title{On two notions of total positivity for partial flag varieties}

\author{Anthony M. Bloch}
\address{Department of Mathematics, University of Michigan}
\email{\href{mailto:abloch@umich.edu}{abloch@umich.edu}}
\author{Steven N. Karp}
\address{Department of Mathematics, University of Notre Dame}
\email{\href{mailto:skarp2@nd.edu}{skarp2@nd.edu}}

\thanks{A.M.B.\ was partially supported by NSF grants DMS-1613819 and DMS-2103026, and AFOSR grant FA 9550-22-1-0215. S.N.K.\ was partially supported by an NSERC postdoctoral fellowship.}

\begin{document}

\begin{abstract}
Given integers $1 \le k_1 < \cdots < k_l \le n-1$, let $\PFl{k_1, \dots, k_l}{n}$ denote the type $A$ partial flag variety consisting of all chains of subspaces $(V_{k_1} \subset \cdots \subset V_{k_l})$ inside $\mathbb{R}^n$, where each $V_k$ has dimension $k$. Lusztig (1994, 1998) introduced the totally positive part $\PFl{k_1, \dots, k_l}{n}^{>0}$ as the subset of partial flags which can be represented by a totally positive $n\times n$ matrix, and defined the totally nonnegative part $\PFl{k_1, \dots, k_l}{n}^{\ge 0}$ as the closure of $\PFl{k_1, \dots, k_l}{n}^{>0}$. On the other hand, following Postnikov (2007), we define $\PFl{k_1, \dots, k_l}{n}^{\Delta > 0}$ and $\PFl{k_1, \dots, k_l}{n}^{\Delta \ge 0}$ as the subsets of $\PFl{k_1, \dots, k_l}{n}$ where all Pl\"{u}cker coordinates are positive and nonnegative, respectively. It follows from the definitions that Lusztig's total positivity implies Pl\"{u}cker positivity, and it is natural to ask when these two notions of positivity agree. Rietsch (2009) proved that they agree in the case of the Grassmannian $\PFl{k}{n}$, and Chevalier (2011) showed that the two notions are distinct for $\PFl{1,3}{4}$. We show that in general, the two notions agree if and only if $k_1, \dots, k_l$ are consecutive integers. We give an elementary proof of this result (including for the case of Grassmannians) based on classical results in linear algebra and the theory of total positivity. We also show that the cell decomposition of $\PFl{k_1, \dots, k_l}{n}^{\ge 0}$ coincides with its matroid decomposition if and only if $k_1, \dots, k_l$ are consecutive integers, which was previously only known for complete flag varieties, Grassmannians, and $\PFl{1,3}{4}$. Finally, we determine which notions of positivity are compatible with a natural action of the cyclic group of order $n$ that rotates the index set.
\end{abstract}

\maketitle

\section{Introduction}\label{sec_introduction}

\noindent A real matrix is called {\itshape totally positive} if all of its minors (i.e.\ determinants of square submatrices) are positive. Totally positive matrices first appeared in the 1930's in work of Schoenberg \cite{schoenberg30} and Gantmakher and Krein \cite{gantmakher_krein37}. In the 1990's, Lusztig \cite{lusztig94} generalized the theory of totally positive matrices to arbitrary semisimple algebraic groups $G$ and their partial flag varieties $G/P$. These spaces have since been widely studied, with connections to representation theory \cite{lusztig94}, combinatorics \cite{postnikov07}, cluster algebras \cite{fomin_williams_zelevinsky}, high-energy physics \cite{arkani-hamed_bourjaily_cachazo_goncharov_postnikov_trnka16, arkani-hamed_bai_lam17}, mirror symmetry \cite{rietsch_williams19}, topology \cite{galashin_karp_lam22}, and many other topics. The purpose of this paper is to examine the definition of the totally positive and totally nonnegative parts of a partial flag variety $G/P$ in type $A$.

\subsection{Lusztig's total positivity vs.\ Pl\"{u}cker positivity}
Given a subset $K = \{k_1 < \cdots < k_l\} \subseteq \{1, \dots, n-1\}$, let $\PFl{K}{n}(\mathbb{R})$ denote the partial flag variety of all tuples of subspaces $(V_{k_i})_{i=1}^l$ of $\mathbb{R}^n$, where $V_{k_1} \subset \cdots \subset V_{k_l}$ and each $V_k$ has dimension $k$. Equivalently, we may view $\PFl{K}{n}(\mathbb{R})$ as a parabolic quotient of $\GL_n(\mathbb{R})$, where a matrix $g\in\GL_n(\mathbb{R})$ represents the flag $(V_{k_i})_{i=1}^l$, where $V_{k_i}$ is the span of the first $k_i$ columns of $g$. Particular cases of interest are when $K = \{1, \dots, n-1\}$, in which case $\PFl{K}{n}(\mathbb{R})$ is the {\itshape complete flag variety} $\Fl_n(\mathbb{R})$, and when $K = \{k\}$, in which case $\PFl{K}{n}(\mathbb{R})$ is the {\itshape Grassmannian} $\Gr_{k,n}(\mathbb{R})$ of all $k$-dimensional subspaces of $\mathbb{R}^n$.

Lusztig \cite{lusztig94,lusztig98} defined the {\itshape totally positive part} $\PFl{K}{n}^{>0}$ of $\PFl{K}{n}(\mathbb{R})$ to be the subset of flags which can be represented by a totally positive matrix in $\GL_n(\mathbb{R})$ (or equivalently, a totally positive lower-triangular unipotent matrix; see \cref{lusztig_definition} for further discussion). He also defined the {\itshape totally nonnegative part} $\PFl{K}{n}^{\ge 0}$ to be the closure of $\PFl{K}{n}^{>0}$. There is an alternative notion of positivity which arises from the {\itshape Pl\"{u}cker embedding} of $\PFl{K}{n}(\mathbb{R})$. Namely, we say that $(V_{k_i})_{i=1}^l$ is {\itshape Pl\"{u}cker positive} if all its Pl\"{u}cker coordinates are positive, or equivalently, if it can be represented by an element of $\GL_n(\mathbb{R})$ whose left-justified (i.e.\ initial) minors of orders $k_1, \dots, k_l$ are all positive. We similarly say that $(V_{k_i})_{i=1}^l$ is {\itshape Pl\"{u}cker nonnegative} if all its Pl\"{u}cker coordinates are nonnegative. We denote the Pl\"{u}cker-positive and Pl\"{u}cker-nonnegative parts of $\PFl{K}{n}(\mathbb{R})$ by $\PFl{K}{n}^{\Delta >0}$ and $\PFl{K}{n}^{\Delta\ge 0}$, respectively. The Pl\"{u}cker-nonnegative part $\Gr_{k,n}^{\Delta\ge 0}$ of the Grassmannian was studied by Postnikov \cite{postnikov07}, and the space $\PFl{K}{n}^{\Delta\ge 0}$ was introduced by Arkani-Hamed, Bai, and Lam \cite[Section 6.3]{arkani-hamed_bai_lam17}, who called it the {\itshape naive nonnegative part}.

We wish to emphasize that both total positivity and Pl\"{u}cker positivity are natural notions. Lusztig's total positivity is compatible with his theory of canonical bases \cite{lusztig90} and with the combinatorics of Coxeter groups \cite{lusztig94}. The space $\PFl{K}{n}^{\ge 0}$ can be decomposed into cells \cite{lusztig94,rietsch99}, each of which admits an explicit parametrization \cite{marsh_rietsch04}, and the cell decomposition forms a regular CW complex \cite{galashin_karp_lam22}. On the other hand, Pl\"{u}cker positivity is more concrete, and leads to connections with matroid theory \cite{ardila_rincon_williams17} and tropical geometry \cite{speyer_williams05}. It also arises in the definition of loop amplituhedra \cite{arkani-hamed_trnka14,bai_he_lam16}, which are spaces appearing in the physics of scattering amplitudes; for example, a {\itshape $1$-loop amplituhedron} is a certain projection of the space $\PFl{k,k+2}{n}^{\Delta \ge 0}$. The construction of loop amplituhedra is incompatible with Lusztig's total positivity, due to the presence of a cyclic symmetry among the particles in the scattering amplitude, as we explain in \cref{sec_cyclic_intro}.

It follows from the definitions that $\PFl{K}{n}^{>0}\subseteq\PFl{K}{n}^{\Delta >0}$ and $\PFl{K}{n}^{\ge 0}\subseteq\PFl{K}{n}^{\Delta\ge 0}$, and it is natural to ask when equality holds. This question has been explored previously in special cases, as we discuss in \cref{sec_history}. We answer this question in general:
\begin{thm}\label{converse}
Let $n\in\mathbb{N}$ and $K\subseteq \{1, \dots, n-1\}$. Then the following are equivalent:
\begin{enumerate}[label=(\roman*), leftmargin=*, itemsep=2pt]
\item\label{converse_tp} $\PFl{K}{n}^{>0} = \PFl{K}{n}^{\Delta > 0}$;
\item\label{converse_tnn} $\PFl{K}{n}^{\ge 0} = \PFl{K}{n}^{\Delta \ge 0}$; and
\item\label{converse_consecutive} the set $K$ consists of consecutive integers.
\end{enumerate}

\end{thm}

We give an elementary proof of \cref{converse}, using classical results in linear algebra and the theory of total positivity. In particular, our proof does not rely on any previously established special cases or on the cell decomposition of $\PFl{K}{n}^{\ge 0}$.

\subsection{Cell decomposition vs.\ matroid decomposition}
Lusztig \cite[Remark 8.15]{lusztig94} introduced a decomposition of $\PFl{K}{n}^{\ge 0}$, which Rietsch \cite{rietsch99} showed is a cell decomposition (i.e.\ each stratum is homeomorphic to an open ball). It is the intersection of the {\itshape projected Richardson stratification} \cite{kazhdan_lusztig79} with $\PFl{K}{n}^{\ge 0}$. There is another natural stratification of $\PFl{K}{n}(\mathbb{R})$, introduced by Gelfand and Serganova \cite{gel'fand_serganova87}, which is the common refinement of the vanishing and nonvanishing sets of all Pl\"{u}cker coordinates. We call this the {\itshape matroid decomposition}, since each stratum is labeled by a tuple of matroids on the ground set $\{1, \dots, n\}$ (or alternatively, a Coxeter matroid).

Postnikov \cite{postnikov07} studied the matroid decomposition of the Pl\"{u}cker-nonnegative part of $\Gr_{k,n}(\mathbb{R})$, which he called the {\itshape positroid decomposition}. He showed that it forms a cell decomposition \cite[Theorem 3.5]{postnikov07}, and that it coincides with the decomposition of Lusztig and Rietsch \cite[Theorem 3.8]{postnikov07} (see \cref{sec_history} for further discussion). The topology of this decomposition presents a stark contrast to the matroid decomposition of all of $\Gr_{k,n}(\mathbb{R})$, which exhibits a phenomenon known as {\itshape Mn\"{e}v universality} \cite{mnev88}. In general, Tsukerman and Williams \cite[Section 7]{tsukerman_williams15} showed that the cell decomposition of $\PFl{K}{n}^{\ge 0}$ is a refinement of the matroid decomposition. They also showed that the two decompositions coincide for complete flag varieties, but differ for $\PFl{1,3}{4}^{\ge 0}$. We determine in general when the two decompositions are equal:
\begin{thm}\label{decompositions}
Let $n\in\mathbb{N}$ and $K\subseteq \{1, \dots, n-1\}$. Then the cell decomposition of $\PFl{K}{n}^{\ge 0}$ coincides with the matroid decomposition if and only if $K$ consists of consecutive integers.
\end{thm}

The forward direction of \cref{decompositions} follows from \cref{converse}, using a topological argument. However, we do not know how to deduce the reverse direction directly from \cref{converse}. Instead, we use a result of Tsukerman and Williams \cite[Theorem 7.1]{tsukerman_williams15}, which in turn builds on unpublished work of Marsh and Rietsch. It states that each cell of $\PFl{K}{n}^{\ge 0}$ is contained in a single matroid stratum, which is uniquely determined by the $0$-dimensional cells in its closure. These $0$-dimensional cells have an explicit description in terms of the Bruhat order on the symmetric group $\mathfrak{S}_n$, due to Rietsch \cite{rietsch06b}. To prove the reverse direction of \cref{decompositions}, we use the combinatorics of $\mathfrak{S}_n$ to reduce the statement to the Grassmannian case, which was proved by Postnikov as described above. Our techniques also have implications in the study of the {\itshape Bruhat interval polytopes} of Kodama and Williams \cite{kodama_williams15} and Tsukerman and Williams \cite{tsukerman_williams15}; see \cref{decompositions_remark} and \cref{minkowski_sum}.

\subsection{Cyclic symmetry}\label{sec_cyclic_intro}

For $\epsilon\in\mathbb{Z}/2\mathbb{Z}$, define the {\itshape (signed) left cyclic shift map} $\sigma_\epsilon\in\GL_n(\mathbb{R})$ by
$$
\sigma_\epsilon(v_1, \dots, v_n) := (v_2, \dots, v_n, (-1)^{\epsilon-1}v_1) \quad \text{ for all } v\in\mathbb{R}^n.
$$
Then $\sigma_\epsilon$ also acts on $\PFl{K}{n}(\mathbb{R})$. If all elements of $K$ have the same parity $\epsilon$, then by the alternating property of the determinant, $\sigma_\epsilon$ acts on Pl\"{u}cker coordinates by rotating the index set $\{1, \dots, n\}$. Therefore $\sigma_\epsilon$ preserves both $\PFl{K}{n}^{\Delta >0}$ and $\PFl{K}{n}^{\Delta\ge 0}$. It is natural to wonder whether $\sigma_\epsilon$ also preserves $\PFl{K}{n}^{>0}$ and $\PFl{K}{n}^{\ge 0}$. One motivation is that the cyclic symmetry for $\PFl{k,k+2}{n}^{\Delta \ge 0}$ is important in the definition of loop amplituhedra mentioned above, coming from a cyclic ordering on the $n$ external particles.

If $K = \{k\}$ and $k$ has parity $\epsilon$, then $\sigma_\epsilon$ preserves both $\Gr_{k,n}^{>0}$ and $\Gr_{k,n}^{\ge 0}$, using \cref{converse} for $\Gr_{k,n}(\mathbb{R})$. We show, however, that in all other cases, $\sigma_\epsilon$ does not preserve either $\PFl{K}{n}^{>0}$ or $\PFl{K}{n}^{\ge 0}$. In particular, one cannot substitute the notion of total positivity for Pl\"{u}cker positivity in the definition of loop amplituhedra.
\begin{thm}\label{cyclic}
Let $K\subseteq \{1, \dots, n-1\}$ such that $|K| \ge 2$, and let $\epsilon\in\mathbb{Z}/2\mathbb{Z}$. Then there exists $V\in\PFl{K}{n}^{>0}$ such that $\sigma_\epsilon(V)\notin\PFl{K}{n}^{\ge 0}$. In particular, $\sigma_\epsilon$ does not preserve $\PFl{K}{n}^{>0}$ or $\PFl{K}{n}^{\ge 0}$.
\end{thm}

While \cref{cyclic} does not follow directly from \cref{converse}, we use similar techniques to prove both results.

One of our initial motivations for this work was to understand which gradient flows on $\PFl{K}{n}(\mathbb{R})$ are compatible with total positivity, which we discuss in a separate paper \cite{bloch_karp1}. We discovered that in certain cases, the classification of such flows differs depending on whether one works with Lusztig's total positivity or with Pl\"{u}cker positivity; see Theorem 5.14, Theorem 5.18, and Remark 5.20 of \cite{bloch_karp1}. \cref{cyclic} above may be regarded as an infinitesimal analogue of this phenomenon.

\subsection{History}\label{sec_history}
We discuss previous work related to \cref{converse} and \cref{decompositions}. Much of this work has focused on the case when $\PFl{K}{n}(\mathbb{R})$ is a Grassmannian $\Gr_{k,n}(\mathbb{R})$, due to Postnikov's study of the matroid decomposition of $\Gr_{k,n}^{\Delta\ge 0}$ \cite{postnikov07}.

Rietsch \cite[Section 4.5]{rietsch98} (also announced in \cite[Section 3.12]{lusztig98}) stated that Lusztig's notion of total positivity coincides with Pl\"{u}cker positivity for all partial flag varieties $G/P$, but her proof contained an error. She later proved \cref{converse} for Grassmannians $\Gr_{k,n}(\mathbb{R})$ in unpublished notes \cite{rietsch09}. Subsequent proofs were given by Talaska and Williams \cite[Corollary 1.2]{talaska_williams13}, Lam \cite[Remark 3.8]{lam16}, and Lusztig \cite{lusztig}. \cref{decompositions} for the complete flag variety $\Fl_n(\mathbb{R})$ was stated in \cite[(9.15)]{galashin_karp_lam22}, and proofs were given by Lusztig \cite[p.\ 4]{lusztig19} and Boretsky \cite[Theorem 5.25]{boretsky}. Conversely, Chevalier \cite[Example 10.1]{chevalier11} gave an example showing that $\PFl{1,3}{4}^{>0} \neq \PFl{1,3}{4}^{\Delta >0}$; see \cite[Remark 5.17]{knutson_lam_speyer13} for a related discussion. In a different direction, Geiss, Leclerc, and Schr\"{o}er \cite[Conjecture 19.2]{geiss_leclerc_schroer08} have conjectured an algebraic description of the totally positive part of any partial flag variety $G/P$.

\cref{decompositions} was proved in the case of Grassmannians $\Gr_{k,n}^{\ge 0}$ by Postnikov \cite[Corollary 3.8]{postnikov07}. We point out that his proof implicitly uses the fact that $\Gr_{k,n}^{\ge 0} = \Gr_{k,n}^{\Delta\ge 0}$, which was only later proved by Rietsch \cite{rietsch09}, as described above. A subsequent proof was given by Talaska and Williams \cite[Corollary 1.2]{talaska_williams13}, and the result can also be proved using work of Tsukerman and Williams \cite[Section 7]{tsukerman_williams15}. \cref{decompositions} in the cases of $\Fl_n^{\ge 0}$ and $\PFl{1,3}{4}^{\ge 0}$ was proved by Tsukerman and Williams \cite[Theorem 7.1 and Remark 7.3]{tsukerman_williams15}.

\subsection{Further directions}
Lusztig \cite{lusztig94,lusztig98} introduced the totally positive and totally nonnegative part of an arbitrary partial flag variety $G/P$, where $G$ is a semisimple algebraic group $G$ over $\mathbb{R}$ and $P$ is a parabolic subgroup. It would be interesting to study the problems considered in this paper for such $G/P$. We mention that Fomin and Zelevinsky \cite{fomin_zelevinsky00b} have considered analogous problems for the group $G$. In particular, they showed that an element of $G$ is totally positive (respectively, totally nonnegative) in the sense of Lusztig \cite{lusztig94} if and only if all its generalized minors are positive (respectively, nonnegative). We consider only the type $A$ case here both for the sake of concreteness, and because our proofs of \cref{converse} and \cref{cyclic} are elementary. We point out that many of our arguments in \cref{sec_decompositions} used to prove \cref{decompositions} can be applied to general Coxeter groups; however, a key tool in our proof is Postnikov's \cref{decompositions_grassmannian}, which is only known in type $A$.

Lusztig \cite[Theorem 3.4]{lusztig98} showed that for any partial flag variety $G/P$, there exists a positive weight $\lambda$ such that a partial flag is totally positive (respectively, totally nonnegative) if and only if its coordinates with respect to the canonical basis of the irreducible $G$-module with highest weight $\lambda$ are all positive (respectively, nonnegative). He also posed the problem of finding the minimal such $\lambda$. Our \cref{converse} gives a partial answer for type $A$ partial flag varieties $G/P = \PFl{K}{n}(\mathbb{R})$: it implies that the sum of fundamental weights $\sum_{k\in K}\omega_k$ is a valid choice of $\lambda$ (and hence minimal) if and only if $K$ consists of consecutive integers.

As we described above, the combinatorics and topology of the cell decomposition of $\PFl{K}{n}^{\ge 0}$ have been extensively studied. For example, each cell has an explicit parametrization \cite[Section 11]{marsh_rietsch04}, the closure poset is shellable \cite[Theorem 1.2]{williams07}, and the cell decomposition forms a regular CW complex \cite[Theorem 1.1]{galashin_karp_lam22}. In light of \cref{converse} and \cref{decompositions}, it would be interesting to further study the matroid decomposition of $\PFl{K}{n}^{\Delta\ge 0}$. Bai, He, and Lam \cite{bai_he_lam16} studied the case $\PFl{1,3}{n}^{\Delta \ge 0}$. In \cite[Corollary 6.16]{bloch_karp1}, we show that $\PFl{K}{n}^{\Delta >0}$ is homeomorphic to an open ball and $\PFl{K}{n}^{\Delta\ge 0}$ is homeomorphic to a closed ball.

\subsection{Outline}
In \cref{sec_background}, we give some background on partial flag varieties and total positivity. In \cref{sec_converse}, we prove \cref{converse}. In \cref{sec_cyclic}, we prove \cref{cyclic}. In \cref{sec_decompositions}, we prove \cref{decompositions}.

\subsection*{Acknowledgments}
We thank George Lusztig, Lauren Williams, and an anonymous reviewer for helpful comments.

\section{Background}\label{sec_background}

\noindent In this section, we define partial flag varieties and their totally positive and totally nonnegative parts, and recall some classical results in the theory of total positivity. For further details on total positivity, we refer to \cite{gantmaher_krein50, karlin68, lusztig94, fomin_zelevinsky00a, pinkus10, fallat_johnson11}. 

\subsection{Notation}\label{sec_notation}
Let $\mathbb{N} := \{0, 1, 2, \dots\}$. For $n\in\mathbb{N}$, we let $[n]$ denote $\{1, 2, \dots, n\}$, and for $i,j\in\mathbb{Z}$, we let $[i,j]$ denote the interval of integers $\{i, i+1, \dots, j\}$. Given a set $S$ and $k\in\mathbb{N}$, we let $\binom{S}{k}$ denote the set of $k$-element subsets of $S$.

We let $e_1, \dots, e_n$ denote the unit vectors of $\mathbb{R}^n$. We let $\mathbb{P}^n(\mathbb{R})$ denote $n$-dimensional real projective space, defined to be $\mathbb{R}^{n+1}\setminus\{0\}$ modulo multiplication by $\mathbb{R}^\times$. For $\lambda_1, \dots, \lambda_n\in\mathbb{R}$, we let $\Diag{\lambda_1, \dots, \lambda_n}$ denote the $n\times n$ diagonal matrix with diagonal entries $\lambda_1, \dots, \lambda_n$. We let $\GL_n(\mathbb{R})$ denote the group of invertible real $n\times n$ matrices.

Given an $m\times n$ matrix $A$, we let $\transpose{A}$ denote the transpose of $A$. For $1 \le k \le m,n$ and subsets $I\in\binom{[m]}{k}$ and $J\in\binom{[n]}{k}$, we let $\Delta_{I,J}(A)$ denote the determinant of the submatrix of $A$ in rows $I$ and columns $J$, called a {\itshape minor} of $A$ of {\itshape order} $k$. If $J = [k]$, we call $\Delta_{I,J}(A)$ a {\itshape left-justified minor} of $A$. We also let $\sumof{I}$ denote the sum of the elements in $I$.

We will make repeated use of the {\itshape Cauchy--Binet identity} (see e.g.\ \cite[I.(14)]{gantmacher59}): if $A$ is an $m\times n$ matrix, $B$ is an $n\times p$ matrix, and $1 \le k \le m,p$, then
\begin{align}\label{cauchy-binet}
\Delta_{I,J}(AB) = \sum_{K\in\binom{[n]}{k}}\Delta_{I,K}(A)\Delta_{K,J}(B) \quad \text{ for all } I\in\textstyle\binom{[m]}{k} \text{ and } J\in\binom{[p]}{k}.
\end{align}

We now introduce partial flag varieties.
\begin{defn}\label{defn_Fl}
Let $n\in\mathbb{N}$, and let $K = \{k_1 < \cdots < k_l\} \subseteq [n-1]$. Let $\P{K}{n}(\mathbb{R})$ denote the parabolic subgroup of $\GL_n(\mathbb{R})$ of block upper-triangular matrices with diagonal blocks of sizes $k_1, k_2 - k_1, \dots, k_l - k_{l-1}, n - k_l$. We define the {\itshape partial flag variety}
$$
\PFl{K}{n}(\mathbb{R}) := \GL_n(\mathbb{R})/\P{K}{n}(\mathbb{R}).
$$
We identify $\PFl{K}{n}(\mathbb{R})$ with the variety of partial flags of subspaces in $\mathbb{R}^n$
$$
\{V = (V_{k_1}, \dots, V_{k_l}) : 0 \subset V_{k_1} \subset \cdots \subset V_{k_l} \subset \mathbb{R}^n \text{ and } \dim(V_{k_i}) = k_i \text{ for } 1 \le i \le l\}.
$$
The identification sends $g\in\GL_n(\mathbb{R})/\P{K}{n}(\mathbb{R})$ to the tuple $(V_k)_{k\in K}$, where $V_k$ is the span of the first $k$ columns of $g$ for all $k\in K$. More generally, if $A$ is any real matrix with $n$ rows and at least $k_l$ columns such that $V_k$ is the span of the first $k$ columns of $A$ for all $k\in K$, we say that $A$ {\itshape represents $V$}. Note that $\GL_n(\mathbb{R})$ acts on $\PFl{K}{n}(\mathbb{R})$ on the left.

We have the {\itshape Pl\"{u}cker embedding}
\begin{align}\label{plucker_embedding}
\begin{gathered}
\PFl{K}{n}(\mathbb{R}) \hookrightarrow \mathbb{P}^{\left(\hspace*{-1pt}\binom{n}{k_1}-1\right)} \times \cdots \times \mathbb{P}^{\left(\hspace*{-1pt}\binom{n}{k_l}-1\right)}, \\
V \mapsto \Big((\Delta_I(A))_{I\in\binom{[n]}{k_1}}, \dots, (\Delta_I(A))_{I\in\binom{[n]}{k_l}}\Big),
\end{gathered}
\end{align}
where $A$ denotes any matrix representative of $V$. (We can check that the definition does not depend on the choice of $A$.) We call the left-justified minors $\Delta_I(A)$ appearing above the {\itshape Pl\"{u}cker coordinates} of $V\in\PFl{K}{n}(\mathbb{R})$ (also known as {\itshape flag minors}), which we denote by $\Delta_I(V)$. We point out that our Pl\"{u}cker coordinates differ from the {\itshape generalized Pl\"{u}cker coordinates} of Gelfand and Serganova \cite{gel'fand_serganova87}, though each encodes essentially the same data; see \cref{generalized_plucker_remark}.

For any $K'\subseteq K$, we have a projection map
\begin{align}\label{defn_Fl_surjection}
\PFl{K}{n}(\mathbb{R}) \twoheadrightarrow \PFl{K'}{n}(\mathbb{R}), \quad (V_k)_{k\in K} \mapsto (V_k)_{k\in K'}.
\end{align}
The map \eqref{defn_Fl_surjection} retains only the subspaces of a partial flag whose dimensions lie in $K'$.

We mention two instances of $\PFl{K}{n}(\mathbb{R})$ which are of particular interest. If $K = [n-1]$, then $\PFl{K}{n}(\mathbb{R})$ is the {\itshape complete flag variety} of $\mathbb{R}^n$, which we denote by $\Fl_n(\mathbb{R})$. If $K$ is the singleton $\{k\}$, then $\PFl{K}{n}(\mathbb{R})$ is the {\itshape Grassmannian} of all $k$-dimensional subspaces of $\mathbb{R}^n$, which we denote by $\Gr_{k,n}(\mathbb{R})$. We also extend the definition of $\Gr_{k,n}(\mathbb{R})$ to $k=0$ and $k=n$.
\end{defn}

\begin{eg}\label{eg_Fl}
Let $n := 4$ and $K := \{1, 3\}$. Then
$$
\P{1,3}{4}(\mathbb{R}) = \left\{\begin{bmatrix}
\ast & \ast & \ast & \ast \\
0 & \ast & \ast & \ast \\
0 & \ast & \ast & \ast \\
0 & 0 & 0 & \ast
\end{bmatrix}\right\}\subseteq\GL_4(\mathbb{R}) \quad \text{ and } \quad \PFl{1,3}{4}(\mathbb{R}) = \GL_4(\mathbb{R})/\P{1,3}{4}(\mathbb{R}).
$$
We identify $\PFl{1,3}{4}(\mathbb{R})$ with the variety of partial flags $V = (V_1, V_3)$ of subspaces of $\mathbb{R}^4$, where $\dim(V_1) = 1$, $\dim(V_3) = 3$, and $V_1 \subset V_3$.

We can represent a generic partial flag $V\in\PFl{1,3}{4}(\mathbb{R})$ by a matrix of the form
$$
A = \begin{bmatrix}
1 & 0 & 0 \\
a & 1 & 0 \\
b & 0 & 1 \\
c & d & e
\end{bmatrix}, \quad \text{ where } a,b,c,d,e\in\mathbb{R}. 
$$
That is, $V_1$ is spanned by the first column of $A$, and $V_3$ is spanned by all three columns of $A$. Then the Pl\"{u}cker embedding \eqref{plucker_embedding} takes $V$ to
\begin{multline*}
\big((\Delta_1(V) : \Delta_2(V) : \Delta_3(V) : \Delta_4(V)), (\Delta_{123}(V) : \Delta_{124}(V) : \Delta_{134}(V) : \Delta_{234}(V))\big) \\
= \big((1 : a : b : c), (1 : e : -d : -ad+c-be)\big) \in \mathbb{P}^3(\mathbb{R})\times\mathbb{P}^3(\mathbb{R}).\qedhere
\end{multline*}

\end{eg}

\begin{defn}\label{defn_totally_positive_matrix}
We say that a real matrix is {\itshape totally positive} if all its minors are positive. For $n\in\mathbb{N}$, we let $\GL_n^{>0}$ denote the subset of $\GL_n(\mathbb{R})$ of totally positive matrices.
\end{defn}

For example, we have $\GL_2^{>0} = \Big\{\scalebox{0.8}{$\begin{bmatrix}a & b \\ c & d\end{bmatrix}$} : a,b,c,d,ad-bc > 0\Big\}$. We now introduce Lusztig's total positivity and Pl\"{u}cker positivity for partial flag varieties.

\begin{defn}\label{defn_totally_positive}
Let $n\in\mathbb{N}$ and $K\subseteq [n-1]$. Following \cite[Section 8]{lusztig94} and \cite[Section 1.5]{lusztig98}, we define the {\itshape totally positive part} of $\PFl{K}{n}(\mathbb{R})$, denoted by $\PFl{K}{n}^{>0}$, as the image of $\GL_n^{>0}$ inside $\PFl{K}{n}(\mathbb{R}) = \GL_n(\mathbb{R})/\P{K}{n}(\mathbb{R})$. Equivalently, $\PFl{K}{n}^{>0}$ consists of all partial flags which can be represented by a totally positive $n\times n$ matrix. We define the {\itshape totally nonnegative part} of $\PFl{K}{n}(\mathbb{R})$, denoted by $\PFl{K}{n}^{\ge 0}$, as the closure of $\PFl{K}{n}^{>0}$ in the Euclidean topology. Note that for any $K'\subseteq K$, the projection map \eqref{defn_Fl_surjection} restricts to surjections
\begin{align}\label{defn_tnn_Fl_surjections}
\PFl{K}{n}^{>0} \twoheadrightarrow \PFl{K'}{n}^{>0} \quad \text{ and } \quad \PFl{K}{n}^{\ge 0} \twoheadrightarrow \PFl{K'}{n}^{\ge 0}.
\end{align}

We also define the {\itshape Pl\"{u}cker-positive part} of $\PFl{K}{n}(\mathbb{R})$, denoted by $\PFl{K}{n}^{\Delta >0}$, as the subset of partial flags whose Pl\"{u}cker coordinates are all positive (up to rescaling). That is, $\PFl{K}{n}^{\Delta >0}$ consists of all partial flags which can be represented by a matrix $A$ such that all left-justified $k\times k$ minors of $A$ are positive for all $k\in K$. We similarly define the {\itshape Pl\"{u}cker-nonnegative part} $\PFl{K}{n}^{\Delta \ge 0}$ by replacing ``positive'' with ``nonnegative'' above.
\end{defn}

Note that by definition, we have $\PFl{K}{n}^{>0} \subseteq \PFl{K}{n}^{\Delta >0}$ and $\PFl{K}{n}^{\ge 0} \subseteq \PFl{K}{n}^{\Delta \ge 0}$. We also have that $\PFl{K}{n}^{\Delta \ge 0}$ is the closure of $\PFl{K}{n}^{\Delta >0}$; see \cref{topology}\ref{topology_plucker}.

\begin{eg}\label{eg_totally_positive}
We have
\begin{gather*}
\Fl_3^{\Delta >0} = \left\{\begin{bmatrix}1 & 0 & 0 \\ a+c & 1 & 0 \\ bc & b & 1\end{bmatrix} : a,b,c > 0\right\} \quad \text{ and } \quad \Gr_{2,4}^{\Delta >0} = \left\{\begin{bmatrix}1 & 0 \\ a & b \\ 0 & 1 \\ -c & d\end{bmatrix} : a,b,c,d > 0\right\}.\qedhere
\end{gather*}

\end{eg}

\begin{rmk}\label{lusztig_definition}
Lusztig's original definition of $\PFl{K}{n}^{>0}$ is slightly different than the one we give in \cref{defn_totally_positive}, but is equivalent. Namely, let $\N_n(\mathbb{R})$ be the subset of $\GL_n(\mathbb{R})$ of all upper-triangular matrices with $1$'s on the diagonal. We define $\N_n^{>0}$ to be the subset of $\N_n(\mathbb{R})$ of matrices whose minors are all positive, except for those which are zero due to upper triangularity. Let $(\N_n^-)^{>0}$ denote the transpose of $\N_n^{>0}$, and let $\H_n^{>0}$ denote the subset of $\GL_n(\mathbb{R})$ of diagonal matrices with positive diagonal entries. Then Lusztig \cite[Section 8]{lusztig94} defines $\PFl{K}{n}^{>0}$ to be the image of $(\N_n^-)^{>0}$ inside $\PFl{K}{n}(\mathbb{R})$. This is equal to the image of $\GL_n^{>0}$, because
\begin{align}\label{LDU_decomposition}
\GL_n^{>0} = (\N_n^-)^{>0} \cdot \H_n^{>0} \cdot \N_n^{>0},
\end{align}
and $\H_n^{>0}\cdot\N_n^{>0} \subseteq \P{K}{n}(\mathbb{R})$. (In fact, Lusztig takes \eqref{LDU_decomposition} to hold by definition; see \cite[Section 2.12]{lusztig94}.) The decomposition \eqref{LDU_decomposition} is a result of Cryer \cite[Theorem 1.1]{cryer73}; we refer to \cite[Chapter 2]{fallat_johnson11} for further discussion and references.
\end{rmk}

\begin{rmk}\label{proj_tnn_GL_to_Fl}
A real matrix is called {\itshape totally nonnegative} if all its minors are nonnegative. Every totally nonnegative matrix in $\GL_n(\mathbb{R})$ represents a totally nonnegative flag in $\PFl{K}{n}(\mathbb{R})$. However, not every element of $\PFl{K}{n}^{\ge 0}$ is represented by a totally nonnegative matrix, unless $K = \emptyset$. For example, the element $\scalebox{0.8}{$\begin{bmatrix}0 & -1 \\ 1 & 0\end{bmatrix}$}\in\Fl_2^{\ge 0}$ cannot be represented by a totally nonnegative matrix, since the top-left entry of a totally nonnegative matrix is positive.
\end{rmk}

We now recall two classical results from the theory of total positivity.
\begin{lem}\label{identity_perturbation}
There exists a continuous function $f : \mathbb{R}_{\ge 0} \to \GL_n(\mathbb{R})$ such that $f(0) = \I_n$ and $f(t)\in\GL_n^{>0}$ for all $t > 0$.
\end{lem}

\begin{proof}
By \cite[II.3.(24)]{gantmaher_krein50} or \cite[(2)]{whitney52} (cf.\ \cite[Problem V.76]{polya_szego25}), we may take $f(t) := (t^{(i-j)^2})_{1 \le i,j \le n}$. Alternatively, by \cite[Theorem 3.3.4]{karlin68}, we may take $f(t) := \exp(tA)$, where $A$ is any tridiagonal $n\times n$ matrix whose entries immediately above and below the diagonal are positive.
\end{proof}

\begin{thm}[{Fekete \cite{fekete_polya12}; see \cite[Lemma 2.1]{pinkus10}}]\label{fekete}
Let $A$ be an $n\times (k+1)$ matrix, where $n\ge k+1$. Suppose that all left-justified $k\times k$ minors of $A$ are positive, and that all $(k+1)\times (k+1)$ minors of $A$ using consecutive rows are positive. Then all $(k+1)\times (k+1)$ minors of $A$ are positive.
\end{thm}

The following two results are duality and restriction statements for totally nonnegative Grassmannians. See \cite[Section 3]{ardila_rincon_williams16} for closely related results. We note that \cref{perpendicular_pluckers} follows from Jacobi's formula for the matrix inverse; we refer to \cite{karp17} for further discussion and references.
\begin{lem}[{\cite[Lemma 1.11(ii)]{karp17}}]\label{perpendicular_pluckers}
Define the bilinear pairing $\langle\cdot,\cdot\rangle$ on $\mathbb{R}^n$ by
$$
\langle v,w\rangle := v_1w_1 - v_2w_2 + v_3w_3 - \cdots + (-1)^{n-1}v_nw_n.
$$
Given $V\in\Gr_{k,n}(\mathbb{R})$, let $V^\perp := \{w\in\mathbb{R}^n : \langle v,w\rangle = 0 \text{ for all } v\in V\} \in \Gr_{n-k,n}(\mathbb{R})$. Then
$$
\Delta_I(V) = \Delta_{[n]\setminus I}(V^\perp) \quad \text{ for all } I\in\textstyle\binom{[n]}{k}.
$$
In particular, $\cdot^\perp$ defines bijections $\Gr_{k,n}^{\Delta >0}\leftrightarrow\Gr_{n-k,n}^{\Delta >0}$ and $\Gr_{k,n}^{\Delta \ge 0}\leftrightarrow\Gr_{n-k,n}^{\Delta \ge 0}$.
\end{lem}

\begin{eg}\label{eg_perpendicular_pluckers}
Let $V\in\Gr_{2,4}(\mathbb{R})$ be represented by the matrix
$$
\begin{bmatrix}
1 & 0 \\
0 & 1 \\
a & b \\
c & d
\end{bmatrix}.
$$
Then $V^\perp\in\Gr_{2,4}(\mathbb{R})$ is represented by the matrix
\begin{gather*}
\begin{bmatrix}
-a & c \\
b & -d \\
1 & 0 \\
0 & 1
\end{bmatrix}.\qedhere
\end{gather*}

\end{eg}

\begin{lem}\label{tnn_restriction}
Let $V\in\Gr_{k,n}^{\Delta\ge 0}$, and let $m \le n$. Define $W := V \cap \spn(e_1, \dots, e_m)$, and let $d := \dim(W)$. Then $W\in\Gr_{d,m}^{\Delta\ge 0}$.
\end{lem}

\begin{proof}
Take an $m\times d$ matrix $B$ representing $W$. Then there exists an $n\times k$ matrix $A$ representing $V$ of the form
$$
A = \begin{bmatrix}
B & \ast \\
0 & C
\end{bmatrix},
$$
where $C$ is an $(n-m)\times (k-d)$ matrix of rank $k-d$. We may assume that $C$ is in reduced column echelon form, so that $C$ restricted to some subset $J\in\binom{[n-m]}{k-d}$ of rows equals $\I_{k-d}$. Let $J' := \{j+m : j\in J\}$. Then for all $I\in\binom{[m]}{d}$, we have
$$
\Delta_I(B) = \Delta_{I\cup J'}(A), \quad \text{ so } \quad \Delta_I(W) = \Delta_{I\cup J'}(V).
$$
Since $V\in\Gr_{k,n}^{\Delta\ge 0}$, we have $W\in\Gr_{d,m}^{\Delta\ge 0}$.
\end{proof}

\section{Lusztig's total positivity vs.\ Pl\"{u}cker positivity}\label{sec_converse}

\noindent In this section we prove \cref{converse}. Our argument will be based on the preliminary results \cref{topology}, \cref{tp_extension}, and \cref{tnn_counterexample}.
\begin{lem}\label{converse_complete}
For $n\in\mathbb{N}$, we have $\Fl_n^{>0} = \Fl_n^{\Delta>0}$.
\end{lem}

\begin{proof}
We know that $\Fl_n^{>0} \subseteq \Fl_n^{\Delta >0}$. Conversely, let $V\in\Fl_n^{\Delta > 0}$. Then there exists an $n\times n$ matrix $A$ representing $V$ whose left-justified minors are all positive. After performing left-to-right column operations, we may assume that $A$ is lower-triangular. For $t > 0$, define
$$
g := A\Diag{t^{n-1},\dots,t,1}\transpose{A}\in\GL_n(\mathbb{R}),
$$
which also represents $V$. We claim that $g\in\GL_n^{>0}$ for all $t$ sufficiently large, whence $V\in\Fl_n^{>0}$, completing the proof. (In fact, $g\in\GL_n^{>0}$ for all $t > 0$, by \eqref{LDU_decomposition}.) To see this, note that for $1 \le k \le n$ and $I,J\in\binom{[n]}{k}$, by \eqref{cauchy-binet} we have that $\Delta_{I,J}(g)$ equals
$$
\sum_{K\in\binom{[n]}{k}}\Delta_{I,K}(A)\Delta_{J,K}(A)t^{kn-\sumof{K}} \\
= \Delta_{I,[k]}(A)\Delta_{J,[k]}(A)t^{kn - \binom{k+1}{2}} + \text{lower order terms}
$$
as $t\to\infty$.
\end{proof}

\begin{lem}\label{matrix_action}
Let $K\subseteq [n-1]$, and let $g\in\GL_n^{>0}$.
\begin{enumerate}[label=(\roman*), leftmargin=*, itemsep=2pt]
\item\label{matrix_action_lusztig} For all $V\in\PFl{K}{n}^{\ge 0}$, we have $g\cdot V\in\PFl{K}{n}^{>0}$.
\item\label{matrix_action_plucker} For all $V\in\PFl{K}{n}^{\Delta\ge 0}$, we have $g\cdot V\in\PFl{K}{n}^{\Delta >0}$.
\end{enumerate}

\end{lem}

\begin{proof}
Part \ref{matrix_action_plucker} follows from the Cauchy--Binet identity \eqref{cauchy-binet}. For part \ref{matrix_action_lusztig}, by \eqref{defn_tnn_Fl_surjections}, it suffices to prove the result for the complete flag variety (when $K = [n-1]$). Since $\Fl_n^{\ge 0} \subseteq \Fl_n^{\Delta \ge 0}$, this case follows from part \ref{matrix_action_plucker} and \cref{converse_complete}.
\end{proof}

\begin{prop}\label{topology}
Let $K\subseteq [n-1]$.
\begin{enumerate}[label=(\roman*), leftmargin=*, itemsep=2pt]
\item\label{topology_lusztig} $\PFl{K}{n}^{\ge 0}$ is the closure of $\PFl{K}{n}^{>0}$, and $\PFl{K}{n}^{>0}$ is the interior of $\PFl{K}{n}^{\ge 0}$.
\item\label{topology_plucker} $\PFl{K}{n}^{\Delta \ge 0}$ is the closure of $\PFl{K}{n}^{\Delta >0}$, and $\PFl{K}{n}^{\Delta >0}$ is the interior of $\PFl{K}{n}^{\Delta \ge 0}$.
\end{enumerate}

\end{prop}

\begin{proof}
Let $f : \mathbb{R}_{\ge 0} \to \GL_n(\mathbb{R})$ be as in \cref{identity_perturbation}.

\ref{topology_lusztig} By definition, $\PFl{K}{n}^{\ge 0}$ is the closure of $\PFl{K}{n}^{>0}$. Also, $\PFl{K}{n}^{>0}$ is open, so it is contained in the interior of $\PFl{K}{n}^{\ge 0}$. It remains to show that the interior of $\PFl{K}{n}^{\ge 0}$ is contained in $\PFl{K}{n}^{>0}$. To see this, let $V\in\PFl{K}{n}^{\ge 0}\setminus\PFl{K}{n}^{>0}$. We claim that $f(t)^{-1}\cdot V$ is not in $\PFl{K}{n}^{\ge 0}$ for all $t > 0$, whence $V$ is not in the interior of $\PFl{K}{n}^{\ge 0}$. Indeed, if $f(t)^{-1}\cdot V\in\PFl{K}{n}^{\ge 0}$ with $t > 0$, then by \cref{matrix_action}\ref{matrix_action_lusztig} we obtain
$$
V = f(t)\cdot (f(t)^{-1}\cdot V) \in \PFl{K}{n}^{>0},
$$
a contradiction.

\ref{topology_plucker} Note that the closure of $\PFl{K}{n}^{\Delta >0}$ is contained in $\PFl{K}{n}^{\Delta \ge 0}$. Conversely, given $V\in\PFl{K}{n}^{\Delta\ge 0}$, we have $V = \lim_{t\to 0,\hspace*{1pt} t>0}f(t)\cdot V$, and $f(t)\cdot V\in\PFl{K}{n}^{\Delta >0}$ for $t > 0$ by \cref{matrix_action}\ref{matrix_action_plucker}. Therefore $\PFl{K}{n}^{\Delta \ge 0}$ is the closure of $\PFl{K}{n}^{\Delta >0}$. The fact that $\PFl{K}{n}^{\Delta >0}$ is the interior of $\PFl{K}{n}^{\Delta \ge 0}$ follows from a similar argument as in the proof of part \ref{topology_lusztig}.
\end{proof}

\begin{lem}\label{tp_extension}
Let $V\in\Gr_{k,n}^{\Delta >0}$, where $1 \le k \le n-1$.
\begin{enumerate}[label=(\roman*), leftmargin=*, itemsep=2pt]
\item\label{tp_extension_plus} There exists $W\in\Gr_{k+1,n}^{\Delta >0}$ such that $V\subseteq W$.
\item\label{tp_extension_minus} There exists $W\in\Gr_{k-1,n}^{\Delta >0}$ such that $W\subseteq V$.
\end{enumerate}

\end{lem}

\begin{proof}
\ref{tp_extension_plus} Take an $n\times k$ matrix $A$ representing $V$ whose $k\times k$ minors are all positive. Let $B  := \begin{bmatrix}A \hspace*{2pt}|\hspace*{2pt} w\end{bmatrix}$ denote the $n\times (k+1)$ matrix formed by concatenating $A$ and the vector $w\in\mathbb{R}^n$, where we define $w$ as follows. We set $w_1, \dots, w_k := 0$, and for $i = k+1, \dots, n$, we take $w_i > 0$ to be sufficiently large that the minor $\Delta_{[i-k,i],[k+1]}(B)$ is positive. By \cref{fekete}, all $(k+1)\times (k+1)$ minors of $B$ are positive. Therefore we may define $W$ to be the column span of $B$.

\ref{tp_extension_minus} This follows by applying part \ref{tp_extension_plus} to $V^\perp$, using \cref{perpendicular_pluckers}.
\end{proof}

\begin{lem}\label{tnn_counterexample}
Let $V\in\Gr_{k,n}^{\Delta\ge 0}$ and $W\in\Gr_{k+1,n}^{\Delta\ge 0}$ such that $V\subseteq W$. If $e_1 + ce_n\in V$ for some $c\in\mathbb{R}$, then $e_1 \in W$.
\end{lem}

\begin{proof}
If $e_1\in V$, then $e_1\in W$. Now suppose that $e_1\notin V$ and $e_1 + ce_n\in V$ for some $c\in\mathbb{R}$, so that $c\neq 0$ and $e_n\notin V$. Let $A$ denote the $n\times k$ matrix representing $V$ in reduced column echelon form. Let $I\in\binom{[n]}{k}$ index the rows containing the pivot $1$'s of $A$; equivalently, $I$ is lexicographically minimal such that $\Delta_I(V)\neq 0$. Since $e_1 + ce_n\in V$ and $e_n\notin V$, we have $1\in I$ and $n\notin I$, and the first column of $A$ is $e_1 + ce_n$. Therefore $\Delta_I(A) = 1$ and $\Delta_{(I\setminus\{1\})\cup\{n\}}(A) = (-1)^{k-1}c$. Since $V$ is Pl\"{u}cker nonnegative, we get that $(-1)^{k-1}c > 0$.

Now take $w\in W\setminus V$ such that $w_i = 0$ for all $i\in I$. Let $B := \begin{bmatrix}A \hspace*{2pt}|\hspace*{2pt} w\end{bmatrix}$ denote the $n\times (k+1)$ matrix representing $W$, formed by concatenating $A$ and $w$. Then for $i\in [n-1]\setminus I$, there exists $\epsilon\in\{1,-1\}$ such that $\Delta_{I\cup\{i\}}(B) = \epsilon w_i$ and $\Delta_{(I\setminus\{1\})\cup\{i,n\}}(B) = \epsilon(-1)^kcw_i$. Since $W$ is Pl\"{u}cker nonnegative and $(-1)^{k-1}c > 0$, we get that $w_i = 0$. Therefore $w$ is a nonzero scalar multiple of $e_n$. Since $e_1 + ce_n\in W$, we obtain $e_1\in W$.
\end{proof}

\begin{proof}[Proof of \cref{converse}]
\ref{converse_tp} $\Leftrightarrow$ \ref{converse_tnn}: This follows from \cref{topology}.

\ref{converse_consecutive} $\Rightarrow$ \ref{converse_tp}: Suppose that $K = [k,l]$. Recall that $\PFl{K}{n}^{>0}\subseteq\PFl{K}{n}^{\Delta >0}$. Conversely, we must show that given $V = (V_k, \dots, V_l)\in\PFl{K}{n}^{\Delta >0}$, we have $V\in\PFl{K}{n}^{>0}$. By repeatedly applying \cref{tp_extension}, there exist $V_i\in\Gr_{i,n}^{\Delta >0}$ for $i = l+1, \dots, n$ and $i = k-1, \dots, 1$ such that $V_1 \subset \cdots \subset V_{n-1}$. Let $W := (V_1, \dots, V_{n-1})\in\Fl_n(\mathbb{R})$. Then $W\in\Fl_n^{\Delta >0}$, so $W\in\Fl_n^{>0}$ by \cref{converse_complete}. Then by \eqref{defn_tnn_Fl_surjections}, we get $V\in\PFl{K}{n}^{>0}$.

\ref{converse_tnn} $\Rightarrow$ \ref{converse_consecutive}: We prove the contrapositive. Suppose that $K$ does not consist of consecutive integers, so that there exist consecutive elements $k<l$ of $K$ with $l-k \ge 2$. Define the element $V = (V_i)_{i\in K}$ of $\PFl{K}{n}(\mathbb{R})$ as follows:
$$
V_i := 
\begin{cases}
\spn(e_5, e_6, \dots, e_{i+4}), & \text{ if $i < k$}; \\
\spn(e_1 + e_4, e_5, e_6, \dots, e_{k+3}), & \text{ if $i = k$}; \\
\spn(e_1 + e_4, e_2, e_3, e_5, e_6, \dots, e_{i+1}), & \text{ if $i \ge l$}.
\end{cases}
$$
That is, $V$ is represented by the $n\times (n-1)$ matrix
$$
A := \begin{bmatrix}
0 & (-1)^{k-1}B & 0 \\
\I_{k-1} & 0 & 0 \\
0 & 0 & \I_{n-k-3}
\end{bmatrix}, \quad \text{ where } \quad B := \begin{bmatrix}
1 & 0 & 0 \\
0 & 1 & 0 \\
0 & 0 & 1 \\
1 & 0 & 0
\end{bmatrix}.
$$
Note that all left-justified minors of $A$ are nonnegative, except for a certain minor of order $k+1$. Since $k+1\notin K$, we get that $V\in\PFl{K}{n}^{\Delta\ge 0}$.

We claim that $V\notin\PFl{K}{n}^{\ge 0}$, which implies that $\PFl{K}{n}^{\ge 0} \neq \PFl{K}{n}^{\Delta \ge 0}$. Indeed, suppose otherwise that $V\in\PFl{K}{n}^{\ge 0}$. Then by \eqref{defn_tnn_Fl_surjections}, we can extend $V$ to a complete flag $(V_1, \dots, V_{n-1})\in\Fl_n^{\ge 0}$. For $1 \le i \le n-1$, define $W_i := V_i\cap\spn(e_1, e_2, e_3, e_4)$. Let $d_i := \dim(W_i)$, so that $W_i\in\Gr_{d_i,4}^{\Delta\ge 0}$ by \cref{tnn_restriction}. Note that $W_k = \spn(e_1 + e_4)$ and $W_l = \spn(e_1 + e_4, e_2, e_3)$. Since the sequence $d_k, d_{k+1}, \dots, d_l$ increases by $0$ or $1$ at each step, and $d_k = 1$ and $d_l = 3$, there exists $j\in [k,l]$ such that $d_j = 2$. Applying \cref{tnn_counterexample} to $W_k$ and $W_j$, we get $e_1\in W_j$. Since $W_j\subset W_l$, this implies $e_1\in W_l$, a contradiction.
\end{proof}

\section{Cyclic symmetry}\label{sec_cyclic}

\noindent In this section, we prove \cref{cyclic}.
\begin{proof}[Proof of \cref{cyclic}]
We claim that it suffices to construct $W\in\PFl{K}{n}^{\ge 0}$ such that $\sigma_\epsilon(W)\notin\PFl{K}{n}^{\ge 0}$. Indeed, let $f : \mathbb{R}_{\ge 0} \to \GL_n(\mathbb{R})$ be as in \cref{identity_perturbation}, so that by \cref{matrix_action}\ref{matrix_action_lusztig}, we have $f(t)\cdot W\in\PFl{K}{n}^{>0}$ for all $t > 0$. If $\sigma_\epsilon(f(t)\cdot W)\in\PFl{K}{n}^{\ge 0}$ for all $t > 0$, then taking $t \to 0$ we obtain $\sigma_\epsilon(W)\in\PFl{K}{n}^{\ge 0}$, a contradiction. Therefore there exists $t > 0$ such that $\sigma_\epsilon(f(t)\cdot W)\notin\PFl{K}{n}^{\ge 0}$, whence we may take $V := f(t)\cdot W$.

Now we construct such a $W = (W_i)_{i\in K} \in \PFl{K}{n}(\mathbb{R})$. Fix any two elements $k < l$ of $K$. We set
$$
W_i := 
\begin{cases}
\spn(e_3, e_4, \dots, e_{i+2}), & \text{ if $i < k$}; \\
\spn(e_1 + e_2, e_3, e_4, \dots, e_{i+1}), & \text{ if $i \ge k$}.
\end{cases}
$$
That is, $W$ is represented by the $n\times (n-1)$ matrix
$$
A := \begin{bmatrix}
0 & (-1)^{k-1} & 0 \\
0 & (-1)^{k-1} & 0 \\
\I_{k-1} & 0 & 0 \\
0 & 0 & \I_{n-k-1}
\end{bmatrix}.
$$
Note that all left-justified minors of $A$ are nonnegative, so $A$ represents an element of $\Fl_n^{\Delta\ge 0}$. By \cref{converse} we have $\Fl_n^{\Delta\ge 0} = \Fl_n^{\ge 0}$, so by \eqref{defn_tnn_Fl_surjections}, we get $W\in\PFl{K}{n}^{\ge 0}$.

Let $X = (X_i)_{i\in K}$ denote the left cyclic shift $\sigma_\epsilon(W)$. Note that
$$
X_k = \spn((-1)^{\epsilon-1}e_n + e_1, e_2, \dots, e_k) \quad \text{ and } \quad X_l = \spn((-1)^{\epsilon-1}e_n + e_1, e_2, \dots, e_l).
$$
Now proceed by contradiction and suppose that $X\in\PFl{K}{n}^{\ge 0}$. Then by \eqref{defn_tnn_Fl_surjections}, we can extend $X$ to a complete flag $(X_1, \dots, X_{n-1})\in\Fl_n^{\ge 0}$. Applying \cref{tnn_counterexample} to $X_k$ and $X_{k+1}$, we get $e_1\in X_{k+1}$. Since $X_{k+1}\subseteq X_l$, this implies $e_1\in X_l$, a contradiction.
\end{proof}

\section{Cell decomposition vs.\ matroid decomposition}\label{sec_decompositions}

\noindent In this section, we prove \cref{decompositions}. We begin by recalling some background in \cref{sec_decompositions_background}. We then give two proofs of the forward direction of \cref{decompositions} in \cref{sec_decompositions_forward}, and prove the reverse direction in \cref{sec_decompositions_reverse}. Throughout this section, we fix $n\in\mathbb{N}$, and let $W$ denote the symmetric group $\mathfrak{S}_n$ of all permutations of $[n]$. Also, $J$ and $K$ will denote complementary subsets of $[n-1]$.

\subsection{Background on Coxeter combinatorics}\label{sec_decompositions_background}
We recall some background on the combinatorics of the Coxeter group $W = \mathfrak{S}_n$; we refer to \cite{bjorner_brenti05} for further details.
\begin{defn}[{\cite[Chapter 2]{bjorner_brenti05}}]\label{defn_bruhat_order}
For $1 \le i \le n-1$, let $s_i := (i \hspace*{8pt} i+1) \in W$ be the simple transposition which exchanges $i$ and $i+1$, and let $e\in W$ denote the identity permutation. Given $w\in W$, a {\itshape reduced word} $\mathbf{w}$ for $w$ is a word in $s_1, \dots, s_{n-1}$ of minimal length whose product is $w$. Each reduced word for $w$ has the same number of letters, called the {\itshape length $\ell(w)$} of $w$, which is equal to the number of inversions of $w$. Any two reduced words for $w$ are related by a sequence of moves of the following form:
\begin{enumerate}[label={(M\arabic*)}, leftmargin=36pt, itemsep=2pt]
\item\label{move_commutation} $s_is_j = s_js_i$ for $1 \le i,j \le n-1$ with $|i-j| \ge 2$; and
\item\label{move_braid} $s_is_{i+1}s_i = s_{i+1}s_is_{i+1}$ for $1 \le i \le n-2$.
\end{enumerate}
In particular, if $s_i$ appears in some reduced word for $w$, then it appears in every reduced word for $w$.

The {\itshape (strong) Bruhat order $\le$} on $W$ is defined as follows: $v \le w$ if and only if for some (or equivalently, for every) reduced word $\mathbf{w}$ for $w$, there exists a reduced word $\mathbf{v}$ for $v$ which is a subword of $\mathbf{w}$. The Bruhat order is graded with rank function $\ell$. The Bruhat order on $W = \mathfrak{S}_3$ is shown in \cref{figure_S3}.
\end{defn}

\begin{figure}[ht]
\begin{center}
$$
\begin{tikzpicture}[baseline=(current bounding box.center),scale=1.0]
\pgfmathsetmacro{\s}{1.0};
\pgfmathsetmacro{\hd}{1.20};
\pgfmathsetmacro{\vd}{0.96};
\pgfmathsetmacro{\is}{1.68};
\node[inner sep=\is](123)at(0,0){\scalebox{\s}{$123$}};
\node[inner sep=\is](213)at($(123)+(-\hd,\vd)$){\scalebox{\s}{$213$}};
\node[inner sep=\is](132)at($(123)+(\hd,\vd)$){\scalebox{\s}{$132$}};
\node[inner sep=\is](312)at($(213)+(0,\vd)$){\scalebox{\s}{$312$}};
\node[inner sep=\is](231)at($(132)+(0,\vd)$){\scalebox{\s}{$231$}};
\node[inner sep=\is](321)at($(312)+(\hd,\vd)$){\scalebox{\s}{$321$}};
\path[semithick](123)edge(213) edge(132) (213)edge(312) edge(231) (132)edge(312) edge(231) (312)edge(321) (231)edge(321);
\end{tikzpicture}
$$
\caption{The Hasse diagram of the Bruhat order on $W = \mathfrak{S}_3$.}
\label{figure_S3}
\end{center}
\end{figure}
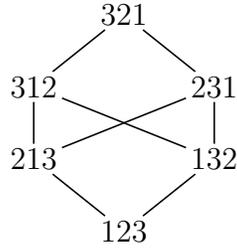

\begin{eg}\label{eg_bruhat_order}
Let $w := 5214763 \in W = \mathfrak{S}_7$. Then $\ell(w) = 9$, and a reduced word for $w$ is $\mathbf{w} = s_1s_3s_4s_3s_2s_1s_5s_6s_5$.
\end{eg}

We will need the following property of reduced words:
\begin{lem}[{\cite[Corollary 1.4.6(ii)]{bjorner_brenti05}}]\label{word_ending}
Let $w\in W$ and $1 \le i \le n-1$. If $\ell(ws_i) < \ell(w)$, then $w$ has a reduced word which ends in $s_i$.
\end{lem}

We define parabolic subgroups and quotients of $W$.
\begin{defn}[{\cite[Section 2.4]{bjorner_brenti05}}]\label{defn_parabolic}
Given $J\subseteq [n-1]$, let $W_J := \langle s_j : j\in J\rangle$ be the subgroup of $W$ generated by the simple transpositions indexed by $J$, called a {\itshape parabolic subgroup}. Equivalently, $W_J$ consists of the elements of $W$ which setwise fix the intervals $[1,k_1], [k_1 + 1, k_2], \dots, [k_l + 1,n]$, where $[n-1]\setminus J = \{k_1 < \cdots < k_l\}$.

Let $W^J$ denote the set of minimal-length coset representatives of the parabolic quotient $W/W_J$. Explicitly, we have
$$
W^J = \{w\in\mathfrak{S}_n : w(j) < w(j+1) \text{ for all } j\in J\}.
$$
Each $w\in W$ has a unique factorization $w = w^Jw_J$ such that $w^J\in W^J$ and $w_J\in W_J$; this factorization is length-additive. In particular, $w^J$ is the minimal-length coset representative of $w$ modulo $W_J$.
\end{defn}

\begin{eg}\label{eg_parabolic_quotient}
Let $w := 5214763 \in W = \mathfrak{S}_7$, as in \cref{eg_bruhat_order}, and let $J := \{1,2,4,6\}$. Then $w^J = 1254736 = s_3s_4s_3s_6s_5$ and $w_J = 3214576 = s_1s_2s_1s_6$.
\end{eg}

We will need the following property of the parabolic factorization:
\begin{lem}[{\cite[Proposition 2.5.1]{bjorner_brenti05}}]\label{parabolic_order}
Let $J\subseteq [n-1]$, and let $v\le w$ in $W$. Then $v^J \le w^J$.
\end{lem}

We now recall the {\itshape Demazure product} and {\itshape downwards Demazure product}, appearing in work of He \cite[Lemma 3.3]{he07} and He and Lu \cite[Appendix A]{he_lu11}. We refer to \cite[Section 2.1]{he_lam15} for further discussion and references.
\begin{defn}[{\cite[Section 1.3]{he09}}]\label{defn_demazure}
There exist binary operations $\ast$ and $\triangleleft$ on $W$ defined by
$$
v\ast w := \max\{vx : x \le w\} \quad \text{ and } \quad v\triangleleft w := \min\{vx : x \le w\}
$$
for all $v,w\in W$. Equivalently,
$$
v \ast (s_{i_1} \cdots s_{i_l}) = (\cdots (v \ast s_{i_1}) \ast \cdots ) \ast s_{i_l} \quad \text{ and } \quad v \triangleleft (s_{i_1} \cdots s_{i_l}) = (\cdots (v \triangleleft s_{i_1}) \triangleleft \cdots ) \triangleleft s_{i_l}
$$
for all $v\in W$ and reduced words $s_{i_1} \cdots s_{i_l}\in W$, where
$$
v\ast s_i = \begin{cases}
vs_i, & \text{ if $\ell(vs_i) > \ell(v)$}; \\
v, & \text{ if $\ell(vs_i) < \ell(v)$}
\end{cases} \quad \text{ and } \quad
v\triangleleft s_i = \begin{cases}
v, & \text{ if $\ell(vs_i) > \ell(v)$}; \\
vs_i, & \text{ if $\ell(vs_i) < \ell(v)$}
\end{cases}
$$
for all $1 \le i \le n-1$. We call $\ast$ the {\itshape Demazure product} and $\triangleleft$ the {\itshape downwards Demazure product}.\footnote{Our operation $\triangleleft$ is the `mirror image' of He's $\triangleright$. We also caution that the symbol $\triangleleft$ is used in \cite{bjorner_brenti05} with a different meaning, namely, to denote a cover relation in the Bruhat order.}
\end{defn}

\begin{eg}\label{eg_demazure}
We have $s_1s_2s_3 \ast s_2s_3s_2 = s_1s_2s_3s_2$ and $s_1s_2s_3 \triangleleft s_2s_3s_2 = s_1$.
\end{eg}

We will need the following property of the Demazure and downwards Demazure products:
\begin{lem}[{\cite[Corollary 1 and Lemma 2]{he09}}]\label{demazure_properites}
Let $v\le w$ in $W$. Then for all $x\in W$, we have $v\ast x \le w\ast x$ and $v\triangleleft x \le w\triangleleft x$.
\end{lem}

\subsection{Background on the cell and matroid decompositions}\label{sec_decompositions_algebra}
We recall the cell decomposition and matroid decomposition of $\PFl{K}{n}^{\ge 0}$, though we will mainly work with \cref{decompositions_equality} and \cref{decompositions_grassmannian}, rather than the definitions. We refer to \cite[Sections 6--7]{tsukerman_williams15} for further details.
\begin{defn}\label{defn_cell_decomposition}
Let $n\in\mathbb{N}$, and let $\B_n(\mathbb{R})$ and $\Bminus_n(\mathbb{R})$ denote the subgroups of $\GL_n(\mathbb{R})$ of upper-triangular and lower-triangular matrices, respectively. For $w\in W$, let $\mathring{w}\in\GL_n(\mathbb{R})$ be any signed permutation matrix corresponding to $w$, i.e., $\mathring{w}_{i,j} = \pm\delta_{i,w(j)}$ for $1 \le i,j \le n$. Given $v,w\in W$ such that $v\le w$, we define the {\itshape (totally nonnegative) Richardson cell}
$$
\cell{v}{w} := (\Bminus_n(\mathbb{R})\cdot\mathring{v}) \cap (\B_n(\mathbb{R})\cdot\mathring{w}) \cap \Fl_n^{\ge 0},
$$
which is the intersection inside $\Fl_n^{\ge 0}$ of the opposite Schubert cell indexed by $v$ and the Schubert cell indexed by $w$.

Now let $J$ and $K$ be complementary subsets of $[n-1]$. Given $v\in W$ and $w\in W^J$ such that $v \le w$, we define the {\itshape (totally nonnegative) projected Richardson cell} $\cell[K]{v}{w} \subseteq \PFl{K}{n}^{\ge 0}$ to be the image of $\cell{v}{w} \subseteq \Fl_n^{\ge 0}$ under the projection map \eqref{defn_tnn_Fl_surjections}. Rietsch \cite{rietsch99,rietsch06b} showed that $\cell[K]{v}{w}$ is homeomorphic to an open ball of dimension $\ell(w) - \ell(v)$. We have the cell decomposition
$$
\PFl{K}{n}^{\ge 0} = \bigsqcup_{\substack{v\in W,\hspace*{2pt} w\in W^J,\\ v\le w}}\cell[K]{v}{w},
$$
where $\PFl{K}{n}^{>0}$ is the unique cell of maximum dimension.
\end{defn}

\begin{rmk}\label{cell_decomposition_remark}
Our definition of the cell decomposition of $\PFl{K}{n}^{\ge 0}$ is different from, but equivalent to, the definition of Rietsch \cite[Section 6]{rietsch06b}. We refer to \cite[Appendix]{he_lam15} and \cite[Remark 4.9]{galashin_karp_lam22} for further discussion.
\end{rmk}

\begin{defn}\label{defn_matroid_decomposition}
Let $K\subseteq [n-1]$. Given a tuple $M = (M_k)_{k\in K}$, where $M_k\subseteq\binom{[n]}{k}$ for $k\in K$, we define
$$
S_M := \{V\in\PFl{K}{n}^{\ge 0} : \text{for all $k\in K$ and $I\in\textstyle\binom{[n]}{k}$, we have } \Delta_I(V) \neq 0 \Leftrightarrow I\in M_k\}.
$$
If $S_M$ is nonempty, we call it a {\itshape (totally nonnegative) matroid stratum}. The {\itshape matroid decomposition} (or {\itshape Gelfand--Serganova decomposition}) of $\PFl{K}{n}^{\ge 0}$ is its decomposition into matroid strata; equivalently, it is the common refinement of the decompositions
$$
\PFl{K}{n}^{\ge 0} = \{V\in\PFl{K}{n}^{\ge 0} : \Delta_I(V)\neq 0\} \sqcup \{V\in\PFl{K}{n}^{\ge 0} : \Delta_I(V)=0\}
$$
for all Pl\"{u}cker coordinates $\Delta_I$.
\end{defn}

\begin{rmk}\label{generalized_plucker_remark}
There is a different, but equivalent, way to define Pl\"{u}cker positivity and the matroid decomposition for partial flag varieties $\PFl{K}{n}(\mathbb{R})$, using the {\itshape generalized Pl\"{u}cker coordinates} of Gelfand and Serganova \cite{gel'fand_serganova87}, rather than the Pl\"{u}cker coordinates of \cref{defn_Fl}. Namely, let $K = \{k_1 < \cdots < k_l\}\subseteq [n-1]$. Given a tuple $I = (I_{k_1}, \dots, I_{k_l})$ such that $I_{k_1} \subset \cdots \subset I_{k_l}$ and $I_k\in\binom{[n]}{k}$ for $k\in K$, define the {\itshape generalized Pl\"{u}cker coordinate}
$$
\Delta_I := \Delta_{I_{k_1}}\Delta_{I_{k_2}} \cdots \Delta_{I_{k_l}}.
$$
Then $V\in\PFl{K}{n}(\mathbb{R})$ is totally positive (respectively, totally nonnegative) if and only if all its generalized Pl\"{u}cker coordinates are positive (respectively, nonnegative). Also, the matroid decomposition of $\PFl{K}{n}^{\ge 0}$ is the common refinement of the decompositions
$$
\PFl{K}{n}^{\ge 0} = \{V\in\PFl{K}{n}^{\ge 0} : \Delta_I(V)\neq 0\} \sqcup \{V\in\PFl{K}{n}^{\ge 0} : \Delta_I(V)=0\}
$$
for all generalized Pl\"{u}cker coordinates $\Delta_I$. These results follow from \cite[Section 9.1]{gel'fand_serganova87} (cf.\ \cite[Chapter 1]{borovik_gelfand_white03}).
\end{rmk}

\begin{rmk}\label{matroid_decomposition_remark}
We use the name {\itshape matroid decomposition} because if $S_M$ is a matroid stratum of $\PFl{K}{n}^{\ge 0}$, then each $M_k$ is a {\itshape (representable) matroid} of rank $k$ on the ground set $[n]$ (in fact, $M_k$ is a {\itshape positroid} \cite{postnikov07,postnikov_speyer_williams09}). Moreover, $M$ itself is a {\itshape (representable) Coxeter matroid}; see \cite[Section 9.1]{gel'fand_serganova87} and \cite[Section 1.7]{borovik_gelfand_white03}.
\end{rmk}

We also recall two results which will be key to our arguments; we refer to \cref{sec_introduction} for further discussion.
\begin{thm}[{Tsukerman and Williams \cite[Theorem 7.1]{tsukerman_williams15}}]\label{decompositions_equality}
Let $J$ and $K$ be complementary subsets of $[n-1]$, and let $v\le w$, where $v\in W$ and $w\in W^J$. Then the cell $\cell[K]{v}{w}$ of $\PFl{K}{n}^{\ge 0}$ is contained in a single matroid stratum, which is uniquely determined by the interval $[v,w]$ modulo $W_J$.
\end{thm}

Note that \cref{decompositions_equality} implies that the cell decomposition of $\PFl{K}{n}^{\ge 0}$ is a refinement of its matroid decomposition.

\begin{thm}[{Postnikov \cite[Theorem 3.8]{postnikov07}, \cite{rietsch09}}]\label{decompositions_grassmannian}
For $0 \le k \le n$, the cell decomposition of $\Gr_{k,n}^{\ge 0}$ coincides with the matroid decomposition.
\end{thm}

\begin{rmk}\label{decompositions_remark}
While it will be sufficient for our purposes to work with the combinatorial statement of \cref{decompositions_equality}, we mention that it has the following geometric interpretation; see \cite[Sections 6--7]{tsukerman_williams15} for further details. The moment polytope of $\PFl{K}{n}(\mathbb{C})$ is a convex polytope in $\mathbb{R}^n$ whose vertices are indexed by $W^J$, or equivalently, by generalized Pl\"{u}cker coordinates of $\PFl{K}{n}(\mathbb{C})$ (see \cref{generalized_plucker_remark}). The moment polytope of $V\in\PFl{K}{n}(\mathbb{C})$ is contained in the moment polytope of $\PFl{K}{n}(\mathbb{C})$, and its vertices correspond precisely to the generalized Pl\"{u}cker coordinates which are nonzero at $V$ \cite[Proposition 5.1]{gel'fand_serganova87}. On the other hand, the set $W^J$ also indexes the zero-dimensional cells of $\PFl{K}{n}^{\ge 0}$, i.e., the cells $\cell[K]{x}{x}$ for $x\in W^J$. If $V\in \cell[K]{v}{w}$, then the zero-dimensional cells in the closure of $\cell[K]{v}{w}$ are precisely $\cell[K]{x^J}{x^J}$ for $x\in [v,w]$ \cite[Theorem 6.1]{rietsch06b}. \cref{decompositions_equality} can be rephrased as saying that the vertices of the moment polytope of $V\in \cell[K]{v}{w}$ are indexed precisely by the zero-dimensional cells in the closure of $\cell[K]{v}{w}$. Implicit in this statement is the fact that the moment polytope of $V$ is equal to the moment polytope of $\cell[K]{v}{w}$, even though the torus orbit of $V$ may have dimension much less than that of $\cell[K]{v}{w}$. This moment polytope is called a {\itshape Bruhat interval polytope}, denoted\footnote{We caution that \cite{tsukerman_williams15} uses the superscript $J$, rather than $K$.} by $\bip[K]{v}{w}$. We will make a further comment about $\bip[K]{v}{w}$ in \cref{minkowski_sum}.
\end{rmk}

\subsection{Proof of the forward direction}\label{sec_decompositions_forward}
In this subsection, we give two proofs of the forward direction of \cref{decompositions}. The first proof uses \cref{converse}, while the second proof uses \cref{decompositions_equality}.

For the first proof, we will need the following result of Rietsch \cite{rietsch98}; see \cite[Corollary 6.16]{bloch_karp1} for a stronger result.
\begin{lem}[{Rietsch \cite[Lemma 5.2]{rietsch98}}]\label{positive_connected}
Let $K\subseteq [n-1]$. Then $\PFl{K}{n}^{\Delta >0}$ is connected.
\end{lem}

\begin{proof}[Proof of the forward direction of \cref{decompositions}]
We prove the contrapositive. Suppose that $K$ does not consist of consecutive integers, so that by the implication \ref{converse_tp} $\Rightarrow$ \ref{converse_consecutive} of \cref{converse}, $\PFl{K}{n}^{>0}$ is strictly contained in $\PFl{K}{n}^{\Delta >0}$. By \cref{positive_connected}, $\PFl{K}{n}^{>0}$ is not closed in $\PFl{K}{n}^{\Delta >0}$. Hence there exists a point $V\in (\PFl{K}{n}^{\ge 0}\setminus\PFl{K}{n}^{>0})\cap\PFl{K}{n}^{\Delta >0}$. Then $V$ and the cell $\PFl{K}{n}^{>0}$ of $\PFl{K}{n}^{\ge 0}$ are contained in the same matroid stratum, namely, the one where all Pl\"{u}cker coordinates are nonzero.
\end{proof}

We now proceed to the second proof of the forward direction of \cref{decompositions}. It is based on the following lemma, which generalizes an example of Tsukerman and Williams \cite[Remark 7.3]{tsukerman_williams15}.
\begin{lem}\label{equal_cells}
Let $J := [2,n-2]$ and $K := \{1,n-1\}$, and let $w := (1 \hspace*{8pt} n)\in W^J$. Then for all $j\in J$, the intervals $[e,w]$ and $[s_j,w]$ are equal modulo $W_J$.
\end{lem}

\begin{proof}
Consider the reduced word
$$
\mathbf{w} := s_1s_2 \cdots s_{n-2}s_{n-1}s_{n-2} \cdots s_2s_1
$$
for $w$. In particular, we see that for $j\in J$, we indeed have $s_j \le w$. Note that $[s_j,w]\subseteq [e,w]$. Conversely, we must show that given $x\in [e,w]$, there exists $y\in [s_j,w]$ such that $x$ and $y$ are equal modulo $W_J$. Take a subword $\mathbf{x}$ of $\mathbf{w}$ which is a reduced word for $x$. We will construct a reduced subword $\mathbf{y}$ of $\mathbf{w}$ which contains $s_j$, and such that the associated permutation $y$ is equal to $x$ modulo $W_J$.

If $\mathbf{x}$ contains $s_j$, we set $\mathbf{y} := \mathbf{x}$. Now suppose that $\mathbf{x}$ does not contain $s_j$. Note that $\mathbf{w}$ contains two occurrences of $s_{j-1}$. Since $\mathbf{x}$ does not contain $s_j$, it does not contain both occurrences of $s_{j-1}$, since otherwise we could use moves \ref{move_commutation} to obtain $s_{j-1}^2$, contradicting the fact that $\mathbf{x}$ is reduced. Similarly, if $\mathbf{x}$ contains the second occurrence of $s_{j-1}$ in $\mathbf{w}$, we may replace it with the first occurrence of $s_{j-1}$. Now let $\mathbf{y}$ be obtained from $\mathbf{x}$ by including the second occurrence of $s_j$ in $\mathbf{w}$. Since $\mathbf{y}$ does not contain the second occurrence of $s_{j-1}$ in $\mathbf{w}$, we can use \ref{move_commutation} to move $s_j$ to the end of $\mathbf{y}$. That is, $y = xs_j$. This implies that $\mathbf{y}$ is reduced; otherwise, by \cref{word_ending}, $x$ would have a reduced word ending in $s_j$, whereas $\mathbf{x}$ (and hence every reduced word for $x$) does not contain $s_j$. Since $s_j\in W_J$, we see that $y$ equals $x$ modulo $W_J$.
\end{proof}

\begin{proof}[Proof of the forward direction of \cref{decompositions}]
We prove the contrapositive. Suppose that $K$ does not consist of consecutive integers, so that there exist consecutive elements $k<l$ of $K$ with $l-k \ge 2$. Let $w := (k \hspace*{8pt} l)\in W^J$. Then by \cref{equal_cells}, for all $j\in [k+1,l-1]$, the intervals $[e,w]$ and $[s_j,w]$ are equal modulo $W_J$. Hence by \cref{decompositions_equality}, the cells $\cell[K]{e}{w}$ and $\cell[K]{s_j}{w}$ of $\PFl{K}{n}^{\ge 0}$ are contained in the same matroid stratum.
\end{proof}

\subsection{Proof of the reverse direction}\label{sec_decompositions_reverse}
In this subsection, we prove the reverse direction of \cref{decompositions}. We first establish two preliminary results, which will allow us to reduce the proof to \cref{decompositions_grassmannian}.
\begin{lem}\label{demazure_reduction}
Let $v\le w$ in $W$, and let $J\subseteq [n-1]$. Set $v' := v\triangleleft w_J^{-1}\in W$ and $w' := w^J\in W^J$. Then $v' \le w'$, and the intervals $[v,w]$ and $[v',w']$ are equal modulo $W_J$.
\end{lem}

\begin{proof}
Note that since the factorization $w = w^Jw_J$ is length-additive, we have $w = w'\ast w_J$ and $w' = w\triangleleft w_J^{-1}$. In particular, $v' \le w'$ by \cref{demazure_properites}.

First we show that given $x'\in [v',w']$, there exists $x\in [v,w]$ such that $x$ and $x'$ are equal modulo $W_J$. We set $x := x' \ast w_J$. Since $w_J\in W_J$, we see that $x$ equals $x'$ modulo $W_J$. Also, by \cref{demazure_properites}, we have
$$
v \le v'\ast w_J \le x'\ast w_J \le w'\ast w_J = w,
$$
so $x\in [v,w]$.

Conversely, we show that given $x\in [v,w]$, there exists $x'\in [v',w']$ such that $x$ and $x'$ are equal modulo $W_J$. We set $x' := x\triangleleft w_J^{-1}$. Since $w_J^{-1}\in W_J$, we see that $x'$ equals $x$ modulo $W_J$. Also, $x'\in [v',w']$ by \cref{demazure_properites}.
\end{proof}

\begin{eg}\label{eg_demazure_reduction}
As in \cref{eg_parabolic_quotient}, we let $w := s_1s_3s_4s_3s_2s_1s_5s_6s_5$ and $J := \{1,2,4,6\}$. Take $v := s_1s_4s_3s_2s_1s_5$, so that $v \le w$. We set
$$
v' := v\triangleleft w_J^{-1} = s_1s_4s_3s_2s_1s_5 \triangleleft s_6s_1s_2s_1 = s_4s_3s_5
$$
and $w' := w^J = s_3s_4s_3s_6s_5$. Then \cref{demazure_reduction} asserts that the intervals $[v,w]$ and $[v',w']$ are equal modulo $W_J$. Indeed, we can verify that both intervals modulo $W_J$ are equal to
$$
\{s_4s_3s_5, s_3s_4s_3s_5, s_4s_3s_6s_5, s_3s_4s_3s_6s_5\},
$$
where above, we represent equivalence classes by elements of $W^J$.
\end{eg}

\begin{lem}\label{reduction_to_grassmannian}
Let $J$ and $K$ be complementary subsets of $[n-1]$, and let $v_1, v_2\le w$, where $v_1, v_2\in W$ and $w\in W^J$. Suppose that the intervals $[v_1,w]$ and $[v_2,w]$ are equal modulo $W_J$. Then $v_1(i) = v_2(i)$ for all $i \le \min(K)$ and all $i \ge \max(K) + 1$.
\end{lem}

\begin{proof}
We prove the statement for $i \le \min(K)$; the statement for $i \ge \max(K) + 1$ follows by symmetry. Set $k := \min(K)$ and $J' := [n-1]\setminus\{k\}$, and note that $J \subseteq J'$. As in \cref{demazure_reduction}, we define $v_1' := v_1\triangleleft w_{J'}^{-1}$, $v_2' := v_2\triangleleft w_{J'}^{-1}$, and $w' := w^{J'}$, so that the intervals $[v_1,w]$, $[v_2,w]$, $[v_1',w']$, and $[v_2',w']$ are all equal modulo $W_{J'}$. By \cref{decompositions_grassmannian} (using \cref{decompositions_equality}), we obtain $v_1' = v_2'$. Now since $w\in W^J$, we have $w(1) < \cdots < w(k)$, so $w_{J'}$ is contained in the parabolic subgroup $W_{[k+1,n-1]}$. Hence $v_1(i) = v_1'(i) = v_2'(i) = v_2(i)$ for all $i \le k$.
\end{proof}

\begin{proof}[Proof of the reverse direction of \cref{decompositions}]
Suppose that $K = [k,l]$. Let $v_1 \le w_1$ and $v_2 \le w_2$, where $v_1,v_2\in W$ and $w_1,w_2\in W^J$. By \cref{decompositions_equality}, it suffices to show that if the intervals $[v_1,w_1]$ and $[v_2,w_2]$ are equal modulo $W_J$, then $v_1=v_2$ and $w_1=w_2$. To this end, we regard $[v_1,w_1]$ and $[v_2,w_2]$ modulo $W_J$ as a subset of $W^J$; by \cref{parabolic_order}, this subset has minimum $v_1^J = v_2^J$ and maximum $w_1 = w_2$. In particular, $v_1(i) = v_2(i)$ for all $i\in [k+1,l]$. Therefore it remains to show that $v_1(i) = v_2(i)$ for all $i \le k$ and all $i \ge l+1$. This follows from \cref{reduction_to_grassmannian}.
\end{proof}

\begin{eg}\label{eg_decompositions_reverse}
We show how the argument above can fail when $K$ is {\itshape not} an interval of integers. Take $n := 4$, $J := \{2\}$, $K := \{1,3\}$, and 
$$
v_1 := 1234 = e, \quad v_2 := 1324 = s_2, \quad w := 4231 = s_1s_2s_3s_2s_1.
$$
By \cref{equal_cells} and \cref{decompositions_equality}, $\cell[K]{v_1}{w}$ and $\cell[K]{v_2}{w}$ are contained in the same matroid stratum. In agreement with \cref{reduction_to_grassmannian}, we have $v_1(i) = v_2(i)$ for all $i \le 1$ and all $i\ge 4$. Also, we have $v_1^J = v_2^J = e$, but this does not imply that $v_1 = v_2$.
\end{eg}

\begin{rmk}\label{minkowski_sum}
Recall the Bruhat interval polytopes $\bip[K]{v}{w}$ discussed in \cref{decompositions_remark}. It follows from the theory of Coxeter matroids (namely, \cite[Corollary 1.13.5]{borovik_gelfand_white03}), along with a result of Tsukerman and Williams \cite[Corollary 7.14]{tsukerman_williams15} (cf.\ \cite[Preface]{borovik_gelfand_white03}, \cite[Theorem 6.3]{caselli_d'adderio_marietti21}), that every Bruhat interval polytope for $\PFl{K}{n}(\mathbb{R})$ can be expressed as the Minkowski sum over $k\in K$ of a Bruhat interval polytope for $\Gr_{k,n}(\mathbb{R})$. \cref{demazure_reduction} allows us to write this Minkowski sum explicitly. Namely, for $v\le w$ with $v\in W$ and $w\in W^J$, we have
\begin{align}\label{minkowski_sum_equation}
\bip[K]{v}{w} = \sum_{k\in K}\bip[\{k\}]{v\hspace*{1pt}\triangleleft\hspace*{1pt} w_{[n-1]\setminus \{k\}}^{-1}}{\hspace*{1pt}w^{[n-1]\setminus \{k\}}}.
\end{align}
We point out that the Bruhat interval polytopes for $\Gr_{k,n}(\mathbb{R})$ are known as {\itshape positroid polytopes} \cite{ardila_rincon_williams16}; see \cite[Proposition 2.8]{tsukerman_williams15} for how to formulate \eqref{minkowski_sum_equation} in terms of positroids.

As an illustration of \eqref{minkowski_sum_equation}, we adopt the setup of \cref{eg_decompositions_reverse}, with $v = v_1$. Then \eqref{minkowski_sum_equation} gives
$$
\bip[\{1,3\}]{e}{s_1s_2s_3s_2s_1} = \bip[\{1\}]{e\hspace*{1pt}\triangleleft\hspace*{1pt} (s_2s_3)^{-1}}{s_3s_2s_1} + \bip[\{3\}]{e\hspace*{1pt}\triangleleft\hspace*{1pt} (s_2s_1)^{-1}}{s_1s_2s_3},
$$
or equivalently,
$$
\bip[\{1,3\}]{1234}{4231} = \bip[\{1\}]{1234}{4123} + \bip[\{3\}]{1234}{2341}.
$$
Note that if we had instead taken $v = v_2$, we would have obtained the same Minkowski sum decomposition for $\bip[\{1,3\}]{1324}{4231}$. Indeed, \cref{equal_cells} implies that $\bip[\{1,3\}]{1234}{4231} = \bip[\{1,3\}]{1324}{4231}$.
\end{rmk}

\bibliographystyle{alpha}
\bibliography{ref}

\end{document}